\documentclass[12pt]{article}
\usepackage[utf8]{inputenc}
\usepackage[english]{babel}
\usepackage{amssymb}
\usepackage{amsmath}
\usepackage{amsthm}

\author{Jonathan Breuer and Eyal Seelig \footnote{Institute of Mathematics, The Hebrew University of Jerusalem, Jerusalem, 91904, Israel.
Supported in part by the Israel Science Foundation (Grant No. 399/16) and in part by the United States-Israel Binational Science
Foundation (Grant No. 2014337), Emails: jbreuer@math.huji.ac.il, eyal.seelig@mail.huji.ac.il}}
\title{On the spacing of zeros of paraorthogonal polynomials for singular measures}

\numberwithin{equation}{section}

\theoremstyle{definition}
\newtheorem{defin}{Definition}[section]

\theoremstyle{remark}
\newtheorem*{rem}{Remark}

\theoremstyle{plain}

\newtheorem{thm}{Theorem}[section]

\newtheorem{lem}[thm]{Lemma}

\newcommand{\C}{{\mathbb{C}}}

\newcommand{\N}{{\mathbb{N}}}
\newcommand{\Q}{{\mathbb{Q}}}
\newcommand{\Z}{{\mathbb{Z}}}
\newcommand{\D}{{\mathbb{D}}}

\newcommand{\sss}[2]{\underset{#1}{\overset{#2}{\sum}}}

\newcommand{\eps}{\varepsilon}

\newcommand{\conv}[2]{\underset{#1 \to #2}{\longrightarrow}}

\newcommand{\el}{{\ell}}
\newcommand{\eel}{{(\ell)}}
\newcommand{\eelp}{{(\ell+1)}}
\newcommand{\p}{{\varphi}}

\newcommand{\pl}[1]{{\varphi_{{#1}}^\eel}}
\newcommand{\plp}[1]{{\varphi_{{#1}}^\eelp}}
\newcommand{\pls}[1]{{\varphi_{{#1}}^{\eel^*}}}
\newcommand{\plps}[1]{{\varphi_{{#1}}^{\eelp^*}}}
\newcommand{\eit}{{e^{i\theta}}}
\newcommand{\ei}[1]{{e^{i{#1}}}}
\newcommand{\thn}[1]{{\theta^{(n)}_{#1}}}
\newcommand{\thnk}[1]{{\theta^{(n_k)}_{#1}}}
\newcommand{\K}{{K_n}}
\newcommand{\Kl}{{K_n^\eel}}
\newcommand{\Klp}{{K_n^\eelp}}
\newcommand{\ml}{{\mu^\eel}}
\newcommand{\vlp}{{v_{\el+1}}}
\newcommand{\nl}{{N_\el}}
\newcommand{\nlp}{{N_{\el+1}}}
\newcommand{\qlp}{{B_{\el+1}}}
\newcommand{\ob}{{\ol{b}}}
\newcommand{\ol}[1]{{\overline{#1}}}
\newcommand{\Leb}{{\mathcal{L}}}
\newcommand{\an}{{\alpha_n}}
\newcommand{\hatn}{{\widehat{N}\eel}}

\newcommand{\se}[1]{{ \left\{#1 \right\}_{n=0}^\infty}}
\newcommand{\see}[1]{{ \{#1 \}_{n=0}^\infty}}
\newcommand{\sel}[1]{{ \left\{#1 \right\}_{\el=0}^\infty}}
\newcommand{\seki}[1]{{ \left\{#1 \right\}_{k=0}^{\infty}}}

\newcommand{\severb}{{\se{\an}}}
\newcommand{\sedecay}{{\sel{v_\el}}}
\newcommand{\sesparse}{{\sel{N_\el}}}

\newcommand{\imaginary}[1]{{\text{Im} #1}}

\newcommand{\abrange}{{\{ | \text{Im} z | < \frac{1}{2} \}}}

\newcommand{\parD}{{\partial\D}}

\newcommand{\plnleit}{{\left|\pl{\nl+1}(\eit)\right|}}
\newcommand{\plpkeit}{{\left|\plp{k}(\eit)\right|}}
\newcommand{\plkeit}{{\left|\pl{k}(\eit)\right|}}

\newcommand{\zerodiff}{{\left(\thn{j+1} - \thn{j}\right)}}

\newcommand{\energy}{{z}}
\newcommand{\en}{{\energy}}
\newcommand{\enn}{{\energy'}}
\newcommand{\ennn}{{\energy''}}
\newcommand{\op}[2]{{#1_{#2}(\en)}}
\newcommand{\opt}[2]{{#1_{#2}(\enn)}}
\newcommand{\St}[1]{{\op{S}{#1}}}
\newcommand{\Tr}[1]{{\op{T}{#1}}}
\newcommand{\Stt}[1]{{\opt{S}{#1}}}
\newcommand{\Trr}[1]{{\opt{T}{#1}}}
\renewcommand{\Q}[1]{{Q_{#1}(\en, \enn)}}
\newcommand{\q}[1]{{{#1}^{-1}}}

\newcommand{\dg}{\dagger}

\begin{document}

\maketitle

\begin{abstract}
We prove a lower bound on the spacing of zeros of paraorthogonal polynomials on the unit circle, based on continuity of the underlying
measure as measured by Hausdorff dimensions. We complement this with the analog of the result from \cite{breuer11} showing that clock spacing
holds even for certain singular continuous measures.
\end{abstract}

%%%%%%%%%%%%%%%%%%%%%%%%%%%%%%%%%%%%%%%%%%%%%%%%%%%%%%%%% Introduction %%%%%%%%%%%%%%%%%%%%%%%%%%%%%%%%%%%%%%%%%%%%%%

\section{Introduction}

This paper is concerned with the spacing of zeros of paraorthogonal polynomials on the unit circle (POPUC), and in particular in the
connection between these spacings and continuity properties of the underlying measure. Its purpose is twofold. The first is to describe a very
general observation connecting measure continuity to local zero spacing (which, to the best of our knowledge, is new in the real line case as
well). The observation is that the degree of continuity of the underlying measure, in terms of comparison with $\alpha$-dimensional Hausdorff
measure, implies a lower bound on the local spacing of the zeros. The second aim of this paper is to present the POPUC analog of an example on
the real line \cite{breuer11, weissman} that shows that singular measures may still have strong asymptotic repulsion, implying that
\emph{upper} bounds coming from singularity of the measure are probably more subtle.

To set the stage, let $\mu$ be a probability measure supported on an infinite subset of $\partial \mathbb{D}$ -- the unit circle. We denote
the (normalized) orthogonal polynomials associated with $\mu$ by the sequence $\se{\p_n}$, which is uniquely defined by the fact that $\p_n$
is a polynomial of degree $n$ with a positive leading coefficient and the orthogonality relation
$$\int_{\partial \D} \p_n \ol{\p_m} d \mu = \delta_{nm}.$$
The sequence $\se{\p_n}$ is well known to satisfy the Szeg\H{o} recurrence (see, e.g., \cite{simon09})
	\begin{equation}\label{recrel}
		\p_{n+1}(z) =
		\rho_n^{-1} \left(
		z \p_n(z)
		- \ol{\an} \p^*_n(z)
		\right)
	\end{equation}
	where $\p^*_n(z) = z^n \ol{\p_n(1/\ol{z})}$, $\rho_n = \left(1 - |\an|^2 \right)^{1/2}$, and the sequence $\severb$, known as the sequence
of Verblunsky coefficients, is a sequence of complex numbers inside the open unit disk, which are uniquely determined by the measure $\mu$.
	
Given a sequence of orthogonal polynomials on the unit circle (OPUC), $\se{\p_n}$ as above, and an additional sequence, $\se{\beta_n}$, of
numbers on the unit circle we may define the corresponding sequence of paraorthogonal polynomials through
	\begin{equation}\label{def-para}
		H_n^{(\beta_{n-1})}(z) :=
		z \p_{n-1}(z) -
		\ol{\beta_{n-1}} \p_{n-1}^*(z).
	\end{equation}

Paraorthogonal polynomials, introduced in \cite{Jones}, have received some attention in recent years due to their natural appearance in
various models both inside and outside the realm of orthogonal polynomial theory. These areas include random matrix theory \cite{KN, KS},
quadrature \cite{Golinskii}, electrostatic problems on the circle \cite{Simanek2}, and the computation of numerical ranges of multiplication
operators \cite{MSS}. More importantly in the context of the present paper, the zeros of paraorthogonal polynomials are in a sense, the
`correct' analog of zeros of orthogonal polynomials on the real line (OPRL): while the zeros of $\p_n$ are known to be inside the open unit
disc \cite{simon09}, the zeros of $H_n^{(\beta_{n-1})}$ are known to lie on $\partial\mathbb{D}$ \cite[Section 2.2]{simon09}. In fact, they
are eigenvalues of a unitary truncation of the CMV matrix associated with the Verblunsky coefficients $\severb$ in much the same way as the
zeros of the $n$'th OPRL are the eigenvalues of a self-adjoint truncation of a Jacobi matrix (see \cite[Sections 2.2 and
8.2]{simon09} for details). Other relevant references include \cite{cantero06,Castillo,Martinez,Simanek1,simon06,simon07Rankone,wong07}.
Questions about the asymptotic distribution of these zeros on $\partial \mathbb{D}$ are thus natural and have been studied in various contexts
which we discuss in greater detail below. In this paper we focus on the connection between the continuity of $\mu$ and the local spacing of
these zeros.

As a final preliminary, we remind the reader of the definition of $\alpha$-dimensional Hausdorff measure, $h^\alpha$. Given $0\leq \alpha \leq
1$ and a nonempty set $S \subseteq \partial \mathbb{D}$
\begin{equation} \label{eq:HausdorffMeasure}
h^\alpha(S)=\lim_{\delta \rightarrow 0}\inf_{\delta-\textrm{covers}}\sum_{j=1}^\infty |I_j|^\alpha.
\end{equation}
In this definition, $\{I_j\}_{j=1}^\infty$ is called a $\delta$-cover of $S$ if for each $j$, $I_j$ is an arc of length $|I_j| < \delta$ and
$S \subseteq \cup_{j=1}^\infty I_j$. The infimum is taken over all $\delta$-covers. It is known \cite{rogers} that the above limit exists
(being possibly $\infty$) for any nonempty $S \subseteq \partial \mathbb{D}$ and that the restriction of $h^\alpha$ to Borel subsets defines a
measure. Note that $h^0$ is the counting measure and $h^1$ is the arc-length measure (=Lebesgue measure on $\partial \mathbb{D}$). Moreover, for
any $S \subseteq \partial \mathbb{D}$, there exists a unique $\alpha(S)\in [0,1]$ so that for any $\alpha<\alpha(S)$, $h^\alpha(S)=\infty$ and for any
$\alpha>\alpha(S)$, $h^\alpha(S)=0$. $\alpha(S)$ is known as the Hausdorff dimension of $S$. For more on Hausdorff measures and dimensions see
\cite{rogers}.

In order to present our results we need to label the zeros around a fixed point on the unit circle. Thus, let $\ei{\Theta} \in \partial \D$ be
fixed and let us label the zeros of $H_n^{(\beta_{n-1})}$ in the following way:
\begin{equation}\label{def-zeros-labels}
		\dots <
		\thn{-1}(\Theta) <
		\Theta \leq
		\thn{0}(\Theta) <
		\thn{1}(\Theta) <
		\dots.
	\end{equation}
\begin{rem}
	We omit $\beta_{n-1}$ from the notation for $\thn{j}(\Theta)$
	in order to streamline the presentation.
	The dependence of $\thn{j}(\Theta)$ on $\beta_{n-1}$
	below will be clear from the context.
\end{rem}
\begin{thm} \label{thm:ContLowerBound}
Let $\mu$ be an infinitely supported probability measure on the unit circle and let $\{\beta_n \}_{n=0}^\infty$ be a sequence of numbers
satisfying $|\beta_n|=1$. Then
\begin{enumerate}
\item If $\mu_{\textrm{ac}}$ is the component of $\mu$ that is absolutely continuous w.r.t.\ Lebesgue measure on $\partial \mathbb{D}$, then
    for $\mu_{\textrm{ac}}$-almost every $z=\ei{\Theta}$ we have
		\begin{equation} \label{eq:ACLower}
			\underset{n\to\infty}
			{\limsup}
			\,
			n (\thn{0}(\Theta) - \thn{-1}(\Theta))
			> 0.
		\end{equation}

\item		Fix $\gamma > 1$, and let
		\begin{equation*}
			A
			=
			\left\{
			\ei\Theta \in \partial\D
			\thinspace \middle| \thinspace
			\underset{n\to\infty}
			{\lim \inf}
			\,
			n^\gamma
			\left(
			\thn{0}(\Theta) - \thn{-1}(\Theta)
			\right)
			< \infty
			\right\}
		\end{equation*}

		Then the restricted measure $\mu(A\cap\cdot)$
		is supported
		on a set of Hausdorff dimension
		at most $\frac{2}{1+\gamma}$.

It follows that for any $0<\alpha<1$, if $\mu$ gives zero weight to sets of Hausdorff dimension at most $\alpha$ (in particular, if $\mu$ is
absolutely continuous with respect to $h^{\alpha+\eps}$ for some $\eps>0$), then for $\mu$-a.e.\ $\Theta$
\begin{equation} \nonumber
\lim_{n\rightarrow \infty} n^\gamma \left(\thn{0}(\Theta) - \thn{-1}(\Theta)\right)=\infty
\end{equation}
for $\gamma=\frac{2}{\alpha}-1$.

	\end{enumerate}
	\end{thm}

\begin{rem}
Many examples of absolutely continuous measures exhibit spacing which is known as local clock behavior (see Definition \ref{def:clock} below).
This is a considerably stronger form of repulsion than that exhibited in \eqref{eq:ACLower}. The bound in \eqref{eq:ACLower}, however, is
completely general (nevertheless, note that \cite{last10} conjecture a weak form of clock behavior for $\mu_{\textrm{ac}}$-a.e.\ point for any
measure). The fundamentally new result in Theorem \ref{thm:ContLowerBound} is part 2 which, to the best of our knowledge, is the only existing
result tying Hausdorff continuity of a measure to OP zero spacing.
\end{rem}

\begin{rem}
The analogous result for zeros of orthogonal polynomials on the real line (OPRL) holds as well. It is in fact an immediate consequence of
\cite[Theorem 2.2]{last08} and \cite[Theorem 1.1]{LastSimonAC} and \cite[Corollary 4.2]{last99}. The analogs of \cite[Theorem
1.1]{LastSimonAC} and \cite[Corollary 4.2]{last99} for the unit circle appear essentially in \cite[Chapter 10]{simon09}. As for the unit
circle analog of \cite[Theorem 2.2]{last08}, a discussion in Section 10 of \cite{last08} describes a strategy of proof and a consequence. For
completeness we state and prove the precise analog in Section 2 below, following which we give the proof of Theorem \ref{thm:ContLowerBound}.
\end{rem}

\begin{rem}
As mentioned above, recent years have seen various works studying zero spacing for paraorthogonal polynomials. For random Verblunsky
coefficients, the works \cite{KS, stoiciu} show a transition from Poisson to clock behavior via asymptotic $\beta$-ensemble statistics (indeed
showing, in this particular case, a correlation between measure continuity and local repulsion). From a slightly different perspective, the
papers \cite{Golinskii,last08,Simanek1, Simanek3} study the connection between regularity properties of $\severb$ and these spacings. In
particular \cite{last08} obtain sufficient conditions ensuring clock behavior, whereas \cite{Golinskii,Simanek3} obtain global upper bounds on
the spacing depending on the decay rate of $\severb$. Other works associating properties of $\mu$ and the zeros of the associated
paraorthogonal polynomials include \cite{Castillo,Martinez,simon06,simon07Rankone}.
\end{rem}

A very strong form of local repulsion between the $\thn{j}(\Theta)$ (already mentioned above) is known as \emph{clock behavior}. This is
defined as follows:
\begin{defin} \label{def:clock}
We say there is clock behavior at $e^{i\Theta} \in \partial \mathbb{D}$ if for every $j\in\Z$
\begin{equation*}
					n \left(\thn{j+1}(\Theta)-\thn{j}(\Theta) \right)
					\conv{n}{\infty}
					2 \pi.
				\end{equation*}
\end{defin}

While the term `clock behavior' was originally defined and studied in the context of the unit circle \cite{simon09} (since in this case the
zeros distribute like dials on a clock), it was studied more extensively in the context of the real line, where it was found to be connected
to universality limits of the Christoffel-Darboux (CD) kernel \cite{LubinskyRev,simon08}. Explicitly, the Freud-Levin-Lubinsky Theorem
\cite{freud71,levin08,simon08} says that convergence of the rescaled CD kernel to the sine kernel (aka `bulk universality') implies clock
behavior at the relevant point. Universality limits have been extensively studied mainly because of their connection to the phenomenon of
universality in random matrix theory. In particular, bulk universality (and therefore clock) was shown to occur for generic points in many
cases of absolutely continuous measures on $\mathbb{R}$ (for a review on some of the relevant literature on universality see
\cite{LubinskyRev}).

In light of the above results and discussion, it is natural to wonder whether singularity of $\mu$ implies less regularity of the asymptotic
zero spacing. The example in \cite{breuer11} (see \cite{weissman} for a continuum Schr\"odinger operator analog) shows that the situation in
the case of $\mathbb{R}$ is more subtle. By considering the Jacobi coefficients associated with $\mu$ on $\mathbb{R}$, \cite{breuer11}
presents a family of purely singular measures where bulk universality, and therefore clock behavior, holds at every point of $[-2,2]$. Our
second main result is the unit circle analog of this example.

\begin{thm} \label{thm:SingularClock}
		There exist purely singular continuous measures on the unit circle such that for any sequence $\{\beta_n\}_{n=0}^\infty$ with
$|\beta_n|=1$, for any $e^{i\Theta} \in \partial \mathbb{D}$, and any $j \in \mathbb{Z}$,
		\begin{equation*}
			n \left(\thn{j+1}(\Theta)-\thn{j}(\Theta) \right)
			\conv{n}{\infty}
			2 \pi.
		\end{equation*}
\end{thm}

As in the case of the real line, we construct these examples by considering the associated Verblunsky coefficients and using the fact that the
association of $\mu$ with the sequence $\severb$ is bijective \cite{simon09}. The sequence $\severb$ that we study is sparse in the sense that
the distances between non-zero $\alpha$'s rapidly increase to infinity. We show that for a sparse decaying sequence of Verblunsky coefficients
the associated CD kernel has sine kernel asymptotics and deduce clock behavior. In this we imitate the strategy and technique of
\cite{breuer11}. We note, however, that we introduce a technical simplification that allows us to consider diagonal and non-diagonal elements
of $K_n$ simultaneously.

\begin{rem}
\cite[Theorem 1.4]{zlatos04} shows that
the measures constructed in \cite{breuer11} are
absolutely continuous with respect to 
$h^\alpha$ for every $0 \leq \alpha < 1$.
In a sense, they are as continuous as possible,
while still being singular with respect to Lebesgue measure.
Although we could not find a proof in the literature 
for the analogous case on the unit circle
(i.e. the measures constructed in Theorem \ref{thm:SingularClock}),
because of the similarities between the constructions, 
we suspect it may also be true in our case.
\end{rem}

The rest of the paper is structured as follows. In Section 2, following a few preliminaries, we present the proof of Theorem
\ref{thm:ContLowerBound}. In Section 3 we set the stage for the proof of Theorem \ref{thm:SingularClock} with a short discussion of the CD
kernel and an explicit description of the example to which the theorem pertains. In section 4 we prove Theorem \ref{thm:SingularClock}. The
appendix contains a statement and proof of the unit circle analog of the Freud-Levin-Lubinsky Theorem (that we could not find in the
literature).

%%%%%%%%%%%%%%%%%%%%%%%%%%%%%%%%%%%%%%%%%%%%%%%%%%%%%%%%%%%%%%%%%%%%% Section 2 %%%%%%%%%%%%%%%%%%%%%%%%%%%%%%%%%%%%%%%%%%%%%%%%%%%%

\section{Lower bounds via Hausdorff dimensions}

We begin by describing the connection between the zeros of the CD kernel and the zeros of $H_n$ that will be useful in later sections as well.
The Christoffel-Darboux kernel, $K_n$, associated with $\mu$, is defined by
	\begin{equation*}
		K_n(z,w) =
		\sss{k=0}{n-1}
		\p_k(z) \ol{\p_k(w)}.
	\end{equation*}
For a sequence $\{\beta_n\}_{n=0}^\infty$, let $\se{H_n^{(\beta_{n-1})}}$ be a sequence of paraorthogonal polynomials defined in
\eqref{def-para}. Wong \cite{wong07} proved that for any $n$, and for any zero of $H_n^{(\beta_{n-1})}$, $z_0$, there exists a constant,
$c\neq 0$ so that
	\begin{equation}\label{eq-para-cd}
		H_n^{(\beta_{n-1})}(z) = c (z-z_0) K_n (z,z_0)
	\end{equation}
	Hence $z$ is also a zero of $H_n$ if and only if it is a zero of $K_n(\cdot, z_0)$. Moreover, it follows that all the zeros of $H_n$ are
simple.

Let $z\in\partial\D$. We study the following system of difference equations
	\begin{equation}\label{eq-opuc}
	\begin{split}
		u_{n+1}
		&=
		\rho_n^{-1} (\en u_n - \ol{\alpha_n} u_n^\dg)
		\\
		u_{n+1}^{\dg}
		&=
		\rho_n^{-1} (- \en \alpha_n u_n + u_n^\dg)
	\end{split}
	\end{equation}
	where $\se{\an}$ is a sequence in $\D$, and $\rho_n = \sqrt{1-|\alpha_n|^2}$.
	Writing this in matrix form we get
	\begin{equation}
		\begin{pmatrix}
		u_{n+1} \\
		u_{n+1}^\dg
		\end{pmatrix}
		=
		\rho_n^{-1}
		\begin{pmatrix}
			\en &
			-\ol{\alpha_n} \\
			-\en \alpha_n &
			1
		\end{pmatrix}
		\begin{pmatrix}
		u_{n} \\
		u_{n}^\dg
		\end{pmatrix}.
	\end{equation}
	Denote
	$\vec{u}_n =
	\begin{pmatrix}
		u_{n} \\
		u_{n}^\dg
	\end{pmatrix}$,
	and the $n$th step matrix
	$
		\St{n} =
		\rho_{n-1}^{-1}
		\begin{pmatrix}
			\en &
			-\ol{\alpha_{n-1}} \\
			-\en \alpha_{n-1} &
			1
		\end{pmatrix}
	$.
	Now the equation can be written as
	\begin{equation}\label{eq-OPUC-St}
		\vec{u}_n = \St{n}\vec{u}_{n-1}.	
	\end{equation}
	Moreover, denote the transfer matrix
	$\Tr{n} = S_n S_{n-1} \cdots S_1$,
	so we have
	\begin{equation}\label{eq-OPUC-Tr}
		\vec{u}_n = \Tr{n} \vec{u}_0.
	\end{equation}
	Pick two solutions of \eqref{eq-OPUC-Tr} with orthogonal boundary conditions
	\begin{align*}
		\begin{pmatrix}
			\p_n(z) \\
			\p_n^\dg(z)
		\end{pmatrix}
		& =
		\Tr{n}
		\begin{pmatrix}
			1 \\
			1
		\end{pmatrix} \\
		\begin{pmatrix}
			\psi_n(z) \\
			\psi_n^\dg(z)
		\end{pmatrix}
		& =
		\Tr{n}
		\begin{pmatrix}
			1 \\
			-1
		\end{pmatrix}
	\end{align*}
	so we can write the transfer matrix as
	\begin{equation}\label{eq-tr-pres}
		\Tr{n} =
		\frac{1}{2}
		\begin{pmatrix}
			\p_n(z) + \psi_n(z)
			&
			\p_n(z) - \psi_n(z) \\
			\p_n^\dg(z) + \psi_n^\dg(z)
			&
			\p_n^\dg(z) - \psi_n^\dg(z)
		\end{pmatrix}.
	\end{equation}
		
		These $\p_n$ are indeed the orthogonal polynomials corresponding to the Verblunsky coefficients $\se{\an}$, and $\p_n^\dg=\p_n^*$.
Furthermore, $\psi_n$ are also known as the second-kind orthogonal polynomials. They are actually just orthogonal polynomials with respect to
$\se{-\an}$ and in this case $\psi_n^\dg=-\psi_n^*$.
	
	\begin{lem}\label{lem-variation-of-params}
	 	Let $\see{\vec{w_n}},\see{\vec{w'_n}}$
	 	be two solutions of \eqref{eq-OPUC-Tr}
	 	for parameters $\en, \enn \in \partial\D$ respectively,
	 	having the same boundary conditions $\vec{w}_0 = \vec{w'}_0$.
	 	Then
	 	\begin{equation}\label{eq-lem-Q1}
			w'_n = w_n +
			(\enn - \en)
			\sss{m=0}{n-1}
			\frac{1}{2 \en^{m+1}}
			\left(
			\p_m^\dg(\en) \psi_n(\en) -
			\p_n(\en) \psi_m^\dg(\en)
			\right)
			w'_m
		\end{equation}
		and
		\begin{equation}\label{eq-lem-Q2}
			w'^\dg_n = w_n^\dg +
			(\enn - \en)
			\sss{m=0}{n-1}
			\frac{1}{2 \en^{m+1}}
			\left(
			\p_m^\dg(\en) \psi_n^\dg(\en) -
			\p_n^\dg(\en) \psi_m^\dg(\en)
			\right)
			w'_m.
		\end{equation}
	\end{lem}
	\begin{proof}
		We would like to find a convenient form for the matrix
		$$
		\Q{n} := 
		\q{\Tr{n}} \Trr{n}.
		$$
		Note that 
		\begin{align*}
			\q{\Tr{n}} \Trr{n}
			&= 
			\q{\Tr{n-1}} \q{\St{n}} \Stt{n} \Trr{n-1}
			\\ &=
			\q{\Tr{n-1}} 
			\left(I + \q{\St{n}} \Stt{n} - I \right) 
			\Trr{n-1}
			\\ &=
			\Q{n-1}
			+ 
			\q{\Tr{n-1}}
			\left(\q{\St{n}} \Stt{n} - I \right) 
			\Trr{n-1}.
		\end{align*}
		% Because $\vec{w},\vec{w'}$ share the same initial conditions, $\Tr{0}=\Trr{0}$. Therefore $\Q{0}=I$. 
		Solving the equation above for the boundary condition $\Q{0}=I$, we get 
		\begin{equation*}
			\Q{n} = I + 
			\sss{m=1}{n}
			\q{\Tr{m-1}}
			\left(\q{\St{m}} \Stt{m} - I \right) 
			\Trr{m-1}.
		\end{equation*}
		Multiplying both sides of $\Trr{n}=\Tr{n}\Q{n}$ by $\vec{w'}_0$, we arrive at
		\begin{equation}\label{eq-goal}
			\vec{w'}_n 
			= 
			\vec{w}_n + 
			\sss{m=1}{n}
			\Tr{n}
			\q{\Tr{m-1}}
			(\q{\St{m}} \Stt{m} - I) 
			\vec{w}_{m-1}.
		\end{equation}
		All that is left is to calculate the summand.
		First,
		\begin{align*}
			\q{\St{m}} \Stt{m}
			&=
			\en^{-1} \rho_{m-1}^{-2}
			\begin{pmatrix}
				1 & 
				\ol{\alpha_{m-1}} \\
				\en \alpha_{m-1} & 
				\en
			\end{pmatrix}
			\begin{pmatrix}
				\enn & 
				-\ol{\alpha_{m-1}} \\
				-\enn \alpha_{m-1} & 
				1
			\end{pmatrix}
			\\
			&= 
			I + 
			\frac{\enn-\en}{\en}
			\begin{pmatrix}
				1 & 0 \\
				0 & 0
			\end{pmatrix}.
		\end{align*}
		Here we used the fact that
		$\det \St{k} = \en$
		for every $k \in \N$. 
		It also implies that
		$\det \Tr{m-1} = 
		\det \left( \St{m-1} \cdots \St{1} \right) = 
		\en^{m-1}$,
		so we conclude
		% \begin{equation}
		% 	\frac
		% 	{\p_n^*\psi_n - \p_n\psi_n^*}
		% 	{2}
		% 	=
		% 	det(\Tr{n})
		% 	=
		% 	\en^n.
		% \end{equation}
		\begin{align*}
			\Tr{n} \q{\Tr{m-1}} (\q{\St{m}} \Stt{m} - I)
			=
			\frac{\enn - \en}{2 \en^m}
			\begin{pmatrix}
				\p_{m-1}^\dg(\en) \psi_n(\en) - 
				\p_n(\en) \psi_{m-1}^\dg(\en) 
				& 0 \\
				\p_{m-1}^\dg(\en) \psi_n^\dg(\en) - 
				\p_n^\dg(\en) \psi_{m-1}^\dg(\en) 
				& 0
			\end{pmatrix}.
		\end{align*}
		Plugging this back into \eqref{eq-goal}, we find that
		\begin{align*}
			\vec{w'}_n 
			&= 
			\vec{w}_n + 
			(\enn - \en)
			\sss{m=1}{n}
			\frac{1}{2 \en^m}
			\begin{pmatrix}
				\p_{m-1}^\dg(\en) \psi_n(\en) - 
				\p_n(\en) \psi_{m-1}^\dg(\en) 
				& 0 \\
				\p_{m-1}^\dg(\en) \psi_n^\dg(\en) - 
				\p_n^\dg(\en) \psi_{m-1}^\dg(\en) 
				& 0
			\end{pmatrix}
			\vec{w}_{m-1} \\
			&= 
			\vec{w}_n + 
			(\enn - \en)
			\sss{m=0}{n-1}
			\frac{1}{2 \en^{m+1}}
			\begin{pmatrix}
				\p_m^\dg(\en) \psi_n(\en) - 
				\p_n(\en) \psi_m^\dg(\en) 
				& 0 \\
				\p_m^\dg(\en) \psi_n^\dg(\en) - 
				\p_n^\dg(\en) \psi_m^\dg(\en) 
				& 0
			\end{pmatrix}
			\vec{w}_m.
		\end{align*}
	\end{proof}

	We also require another Lemma.
	\begin{lem}\label{lem-evecs-orthogonality}
		Let $\se{\p_n}$ be a sequence of orthogonal polynomials on the unit circle and $\beta \in \partial \mathbb{D}$. Let $\enn,\ennn$ be two distinct zeros of
the paraorthogonal polynomial $H_n^{(\beta)}$. Then the pair of vectors
		\begin{equation*}
			\begin{pmatrix}
				\p_0(\enn) \\
				\p_1(\enn) \\
				\vdots \\
				\p_{n-1}(\enn)
			\end{pmatrix},
			\begin{pmatrix}
				\p_0(\ennn) \\
				\p_1(\ennn) \\
				\vdots \\
				\p_{n-1}(\ennn)
			\end{pmatrix}
		\end{equation*}
		are orthogonal to each other in the Euclidean space $\C^n$.
	\end{lem}
	\begin{proof}
		As $\beta$ is fixed, we omit it 
		from the notation for $H_n^{(\beta)}$
		throughout the proof.
		Let $U:L^2(\mu)\to L^2(\mu)$
		be the operator of multiplication by $z$.
		Let $P_n:L^2(\mu)\to L^2(\mu)$
		be the oblique projection operator
		into $span\{z^m\}_{m=0}^{n-1}$
		along $span\{H_n, \p_{n+1}, \p_{n+2}, \dots\}$.
		Now define another operator
		\begin{equation*}
			U_n := P_n U \vert_{span\{z^m\}_{m=0}^{n-1}}.
		\end{equation*}
		As discussed in \cite{cantero06}, this is the appropriate way
		to truncate the unitary operator $U$ in order to
		get a unitary operator on a finite-dimensional
		subspace of $L^2(\mu)$.
		This operator acts on ${span\{z^m\}_{m=0}^{n-1}}$ by
		\begin{equation*}
			(U_n f)(z)
			=
			zf(z)
			-
			\frac
			{\left<
			zf, \p_n
			\right>}
			{\left<
			H_n, \p_n
			\right>}
			H_n(z)
		\end{equation*}
		Let $\lambda$ be an eigenvalue of $U_n$,
		then there exists an eigenfunction
		$f \in span\{z^m\}_{m=0}^{n-1}$
		such that
		\begin{align*}
			zf(z)
			-
			\frac
			{\left<
			zf, \p_n
			\right>}
			{\left<
			H_n, \p_n
			\right>}
			H_n(z)
			&=
			\lambda f \\
			&\Updownarrow \\
			\frac
			{\left<
			zf, \p_n
			\right>}
			{\left<
			H_n, \p_n
			\right>}
			H_n(z)
			&=
			(z-\lambda) f,
		\end{align*}
		so $\lambda$ is a zero of $H_n$.
		Moreover, $f$ lies in the
		one-dimensional space spanned by the function
		$\frac{H_n(z)}{z-\lambda}$,
		so every eigenvalue is simple.
		Thus the set of eigenvalues of $U_n$
		equals the set of zeros of $H_n$.
		In particular, $\enn, \ennn$ are
		two distinct eigenvalues of $U_n$.

		By \eqref{eq-para-cd},
		there exist constants $0\neq c_1,c_2 \in \C$
		such that
		\begin{align*}
			\frac
			{H_n(z)}
			{z-\enn}
			&=
			c_1 \K(z, \enn) \\
			\frac
			{H_n(z)}
			{z-\ennn}
			&=
			c_2 \K(z, \ennn),
		\end{align*}
		so $\K(z,\enn), \K(z,\ennn)$
		are eigenfunctions of $U_n$
		associated with the distinct eigenvalues
		$\enn,\ennn$ respectively.
		Therefore
		$$
		\left<
		\K(z,\enn), \K(z,\ennn)
		\right>_{L^2(\mu)}
		= 0.
		$$
		Taking $B=\{\p_m(z)\}_{m=0}^{n-1}$
		as a basis for $span\{z^m\}_{m=0}^{n-1}$,
		we write the coordinates vector of the eigenfunctions
		\begin{equation*}
			[\K(z,\enn)]_B =
			\begin{pmatrix}
				\ol{\p_0(\enn)} \\
				\vdots \\
				\ol{\p_{n-1}(\enn)}
			\end{pmatrix}
			, \quad
			[\K(z,\ennn)]_B =
			\begin{pmatrix}
				\ol{\p_0(\ennn)} \\
				\vdots \\
				\ol{\p_{n-1}(\ennn)}
			\end{pmatrix}.
		\end{equation*}
		These vectors in $\C^n$ are again
		orthogonal to each other as eigenvectors of $[U_n]_B$
		associated with distinct eigenvalues $\enn,\ennn$.
		We conclude that also their complex conjugates
		$$
			\begin{pmatrix}
				\p_0(\enn) \\
				\vdots \\
				\p_{n-1}(\enn)
			\end{pmatrix}
			,
			\begin{pmatrix}
				\p_0(\ennn) \\
				\vdots \\
				\p_{n-1}(\ennn)
			\end{pmatrix}
		$$
		are orthogonal to each other.
	\end{proof}
	We now use these two lemmas to prove the following unit circle version of \cite[Theorem 2.2]{last08}
	\begin{thm}\label{thm-lower-bound}
		Let $\en=\ei{\Theta}\in\partial\D$, $\beta \in \partial \D$, and let
		\begin{equation}
			\thn{-1}:=\thn{-1}(\Theta) <
			\Theta \leq
			\thn{0}(\Theta)=:\thn{0}	
		\end{equation}
		be as in \eqref{def-zeros-labels},
		i.e. $\enn=\ei{\thn{-1}}$ and $\ennn=\ei{\thn{0}}$
		are a pair of consecutive zeros
		of the paraorthogonal polynomial $H_n^{(\beta)}$
		around $\en$.
		Then
		\begin{equation}
			\left|\thn{0} - \thn{-1}\right|
			\geq
			\left(
			\sss{k=0}{n-1}||\Tr{k}||^2
			\right)^{-1}
		\end{equation}
	\end{thm}
	\begin{proof}
We imitate the proof of \cite[Theorem 2.2]{last08}. Let $\seki{\p_k}$
		be the sequence of normalized OPUC.
		So
		$$
		\seki{\p_k(\en)},
		\seki{\p_k(\enn)},
		\seki{\p_k(\ennn)}
		$$
		solve \eqref{eq-OPUC-Tr} for parameters
		$\en, \enn, \ennn$ respectively.
		Thus, by Lemma \ref{lem-variation-of-params},
		\begin{align*}
			\p_k(\enn)
			&=
			\p_k(\en) +
			(\enn - \en)
			\sss{m=0}{k-1}
			\frac{1}{2 \en^{m+1}}
			\left(
			\p_m^\dg(\en) \psi_k(\en) -
			\p_k(\en) \psi_m^\dg(\en)
			\right)
			\p_m(\enn), \\
			\p_k(\ennn)
			&=
			\p_k(\en) +
			(\ennn - \en)
			\sss{m=0}{k-1}
			\frac{1}{2 \en^{m+1}}
			\left(
			\p_m^\dg(\en) \psi_k(\en) -
			\p_k(\en) \psi_m^\dg(\en)
			\right)
			\p_m(\ennn).
		\end{align*}
		Define the operator $A_n(\en) : \C^n \to \C^n$ by
		\begin{equation*}
			(A_n(\en) v)_k
			=
			\sss{m=0}{k-1}
			\frac{1}{2 \en^{m+1}}
			\left(
			\p_m^\dg(\en) \psi_k(\en) -
			\p_k(\en) \psi_m^\dg(\en)
			\right)
			v_m.
		\end{equation*}
		Thinking now of
		$\p_\cdot(\en),\p_\cdot(\enn),\p_\cdot(\ennn)$
		as vectors in $\C^n$, i.e.
		\begin{equation*}
			\p_\cdot(\xi)
			=
			\begin{pmatrix}
				\p_0(\xi) \\
				\vdots \\
				\p_{n-1}(\xi)
			\end{pmatrix},\quad
			\text{for } \xi=\en,\enn,\ennn,
		\end{equation*}
		we may write
		\begin{align*}
			\p_\cdot(\enn)
			&=
			\p_\cdot(\en)
			+
			(\enn - \en)
			A_n(\en)\p_\cdot(\enn) \\
			\p_\cdot(\ennn)
			&=
			\p_\cdot(\en)
			+
			(\ennn - \en)
			A_n(\en)\p_\cdot(\ennn). \\
		\end{align*}
		Let $||\cdot||_n$ denote the Euclidean norm in $\C^n$.
		Without loss of generality, assume
		$||\p_\cdot(\ennn)||_n
		\leq
		||\p_\cdot(\enn)||_n$.
		Otherwise, switch the roles of $\enn,\ennn$
		in the argument below.

		On the one hand,
		\begin{equation*}
			||\p_\cdot(\enn)||_n^2
			=
			\left<
			\p_\cdot(\enn),
			\p_\cdot(\enn)
			\right>
			=
			\left<
			\p_\cdot(\en),
			\p_\cdot(\enn)
			\right>
			+
			(\enn - \en)
			\left<
			A_n(\en)\p_\cdot(\enn),
			\p_\cdot(\enn)
			\right>,
		\end{equation*}
		and on the other hand, by Lemma \ref{lem-evecs-orthogonality}
		\begin{equation*}
			0
			=
			\left<
			\p_\cdot(\ennn),
			\p_\cdot(\enn)
			\right>
			=
			\left<
			\p_\cdot(\en),
			\p_\cdot(\enn)
			\right>
			+
			(\ennn - \en)
			\left<
			A_n(\en)\p_\cdot(\ennn),
			\p_\cdot(\enn)
			\right>.
		\end{equation*}
		Subtracting these equations and taking the absolute value, we get
		\begin{align*}
			||\p_\cdot(\enn)||_n^2
			&\leq
			|\enn - \en|
			|\left<
			A_n(\en)\p_\cdot(\enn),
			\p_\cdot(\enn)
			\right>| \\
			&\quad
			+
			|\ennn - \en|
			|\left<
			A_n(\en)\p_\cdot(\ennn),
			\p_\cdot(\enn)
			\right>| \\
			&\leq
			\left(
			|\enn - \en| +
			|\ennn - \en|
			\right)
			||A_n(\en)|| \cdot
			||\p_\cdot(\enn)||_n^2 \\
			&\Downarrow \\
			||A_n(\en)||^{-1}
			&\leq
			|\enn - \en| +
			|\ennn - \en|.
		\end{align*}
		Because the distance between two points
		on a circle is smaller than the length of the
		arc connecting them, we see that
		\begin{align*}
			|\enn - \en| +
			|\ennn - \en|
			\leq
			\left|\thn{-1} - \Theta \right|
			+
			\left|\thn{0} - \Theta \right|
			&=
			\left|\thn{0} - \thn{-1} \right|.
		\end{align*}
		All that is left is to show that
		\begin{equation*}
			||A_n(z)||
			\leq
			\sss{k=0}{n-1}
			||\Tr{k}||^2.
		\end{equation*}
		Indeed, we obtain this inequality
		by estimating the Hilbert-Schmidt norm of
		$A_n(z)$, which is larger or equal to its operator norm.
		\begin{align*}
			||A_n(z)||^2_{\operatorname{HS}}
			&=
			\frac{1}{4}
			\sss{k=1}{n-1}
			\sss{m=0}{k-1}
			\left|
			\p_m^\dg(\en) \psi_k(\en) -
			\p_k(\en) \psi_m^\dg(\en)
			\right|^2 \\
			&\leq
			\frac{1}{4}
			\sss{k=0}{n-1}
			\sss{m=0}{n-1}
			\left(
			\left|
			\p_m^\dg(\en) \psi_k(\en)
			\right|^2
			+
			2 \left|
			\p_m^\dg(\en) \psi_k(\en)
			\p_k(\en) \psi_m^\dg(\en)
			\right|
			+
			\left|
			\p_k(\en) \psi_m^\dg(\en)
			\right|^2
			\right).
		\end{align*}
		Summing over each of the three terms separately, 
		one finds that the first term yields 
		$
			\frac{1}{4}
			||\p_\cdot^\dg(z)||_n^2
			||\psi_\cdot(z)||_n^2
		$
		and the last term yields
		$
			\frac{1}{4}
			||\p_\cdot(z)||_n^2
			||\psi_\cdot^\dg(z)||_n^2
		$. 
		As for the middle term, by Cauchy-Schwarz, 
		its sum is less than or equal to
		\begin{equation*}
			\frac{1}{2}
			||\p_\cdot^\dg(z)||_n
			||\psi_\cdot(z)||_n
			||\p_\cdot(z)||_n
			||\psi_\cdot^\dg(z)||_n
			\leq
			\frac{1}{4}
			\left(
			||\p_\cdot^\dg(z)||_n^2
			||\psi_\cdot^\dg(z)||_n^2
			+
			||\p_\cdot(z)||_n^2
			||\psi_\cdot(z)||_n^2
			\right).
		\end{equation*}
		Therefore,
		\begin{equation*}
			||A_n(z)||^2_{\operatorname{HS}}
			\leq
			\frac{1}{4}
			\Big(
			||\p_\cdot(z)||_n^2
			+
			||\p_\cdot^\dg(z)||_n^2
			\Big)
			\Big(
			||\psi_\cdot(z)||_n^2
			+
			||\psi_\cdot^\dg(z)||_n^2
			\Big).
		\end{equation*}
	Note that
	\begin{align*}
		\frac{1}{2}
		\left(||\p_\cdot(z)||_n^2
		+
		||\p_\cdot^\dg(z)||_n^2
		\right)
		&=
		\frac{1}{2}
		\sss{k=0}{n-1}
		\left(
		|\p_k(z)|^2
		+
		|\p_k^\dg(z)|^2
		\right) \\
		&=
		\frac{1}{2}
		\sss{k=0}{n-1}
		\left \| \Tr{k}
		\begin{pmatrix}
			1 \\
			1
		\end{pmatrix}
		\right \|^2 \\
		&\leq
		\sss{k=0}{n-1}
		||\Tr{k}||^2,
	\end{align*}
	and similarly for
	$||\psi_\cdot(z)||_n^2+||\psi_\cdot^\dg(z)||_n^2$,
	only by using
	$
	\Tr{k}
	\begin{pmatrix}
		1 \\
		-1
	\end{pmatrix}
	=
	\begin{pmatrix}
		\psi_k(z) \\
		\psi_k^\dg(z)
	\end{pmatrix}
	$.
	Plugging this back in,
	\begin{equation*}
		||A_n(z)||^2_{\operatorname{HS}}
		\leq
		\left(
		\sss{k=0}{n-1}
		||\Tr{k}||^2
		\right)^2.
	\end{equation*}
	Hence
	\begin{equation*}
		\left|\thn{0} - \thn{-1} \right|
		\geq
		||A_n(z)||^{-1}
		\geq
		||A_n(z)||^{-1}_{\operatorname{HS}}
		\geq
		\left(
		\sss{k=0}{n-1}
		||\Tr{k}||^2
		\right)^{-1},
	\end{equation*}
	which concludes the proof.

	\end{proof}
	\enskip

	\enskip
	\begin{proof}[Proof of Theorem \ref{thm:ContLowerBound}]

Fix a sequence $\{\beta_n \}_{n=0}^\infty$ with $|\beta_n|=1$. 		

\begin{enumerate}
 \item By \cite[Theorem 10.9.4]{simon09} an essential support for $\mu_{\text{ac}}$ is
		\begin{equation*}
			N_1
			=
			\left\{
			z \in \partial\D
			\thinspace \middle| \thinspace
			\underset{n\to\infty}
			{\liminf}
			\,
			\frac{1}{n}
			\sss{k=0}{n-1}
			||\Tr{k}||^2
			< \infty
			\right\}.
		\end{equation*}
		Notice that $z=\ei\Theta \in N_1$
		if and only if
		\begin{equation*}
			\underset{n\to\infty}
			{\limsup}
			\,
			n \left(
			\sss{k=0}{n-1}
			||\Tr{k}||^2
			\right)^{-1}
			> 0,
		\end{equation*}
		which implies,
		by Theorem \ref{thm-lower-bound},
		that
		\begin{equation*}
			\underset{n\to\infty}
			{\limsup}
			\,
			n (\thn{0}(\Theta) - \thn{-1}(\Theta))
			> 0.
		\end{equation*}
	\item By Theorem \ref{thm-lower-bound}, $z \in A$ implies
		\begin{equation*}
			\underset{n\to\infty}
			{\liminf}
			\,
			n^\gamma
			\left(
			\sss{k=0}{n-1}
			||T_k(z)||^2
			\right)^{-1}
			< \infty,
		\end{equation*}
		which is equivalent to
		\begin{equation*}
			\underset{n\to\infty}
			{\limsup}
			\,
			\frac{1}{n^\gamma}
			\sss{k=0}{n-1}
			||T_k(z)||^2
			> 0.
		\end{equation*}
		Now, by \eqref{eq-tr-pres},
 		\begin{align*}
			||T_k(z)||^2
			&\leq
			\frac{1}{4}
			\left(
			|\p_k(z) + \psi_k(z)|^2 +
			|\p_k(z) - \psi_k(z)|^2
			\right. \\
			&\quad
			+ \left.
			|\p_k^\dg(z) + \psi_k^\dg(z)|^2 +
			|\p_k^\dg(z) - \psi_k^\dg(z)|^2
			\right) \\
			&\leq
			|\p_k(z)|^2 +
			|\psi_k(z)|^2.
		\end{align*}
		The last step is due to the fact that 
		\begin{equation}\label{eq:NormEqualityOnCircle}
			\forall z \in \partial\D \quad
			|\varphi_k(z)|=
			|z^k \ol{\varphi_k(1/\ol{z})}|=
			|\varphi^*_k(z)|
		\end{equation}
		and as we noted earlier in this chapter,
		$\p_k^\dg, \psi_k^\dg$ are just $*$-conjugates
		of first- and second-kind orthogonal polynomials
		(up to a sign change).
		Therefore, 
		\begin{equation*}
			\underset{n\to\infty}
			{\limsup}
			\,
			\frac{1}{n^\gamma}
			\left(
			||\p_\cdot(z)||_n^2 +
			||\psi_\cdot(z)||_n^2
			\right)
			> 0.
		\end{equation*}
		Now, by \cite[Theorem 4.3.16]{simon09}
		for $\mu$-a.e.\ $z\in\partial\D$
		and for any $\eta>0$,
		there exists a constant $C_\eta$ such that
		\begin{equation*}
			||\p_\cdot(z)||_n
			\leq
			C_\eta
			\cdot
			n^{\frac{1}{2} + \eta},
		\end{equation*}
		which implies that
		\begin{equation*}
			\frac
			{||\p_\cdot(z)||_n^2}
			{n^\gamma}
			\leq
			C_\eta^2
			\cdot
			n^{1 - \gamma + 2\eta}
		\end{equation*}
		converges to zero as $n$ goes to infinity,
		by choosing $\eta$ small enough.
		
		It follows that
		\begin{equation*}
			\underset{n\to\infty}
			{\limsup}\,
			\frac
			{||\psi_\cdot(z)||_n^2}
			{n^\gamma}
			> 0,
		\end{equation*}
		and we conclude that for $\delta=\frac{1}{\gamma - \eps}$ (where $\gamma>\varepsilon>0$)
		\begin{align*}
			\underset{n\to\infty}
			{\liminf}\,
			\frac
			{||\p_\cdot(z)||_n}
			{||\psi_\cdot(z)||_n^\delta}
			&\leq
			\underset{n\to\infty}
			{\liminf}\,
			C n^{(1 - \delta \gamma) / 2 + \eta}
		\end{align*}
		for some constant $C$, 
		which again converges to zero as $n$ goes to infinity
		by choosing $\eta$ small enough.
		By the subordinacy theory for OPUC
		\cite[Theorems 10.8.5, 10.8.7]{simon09},
		$\mu(A\cap\cdot)$ is supported on a set
		of Hausdorff dimension at most
		$\frac{2 \delta}{1+ \delta} = \frac{2}{1+\gamma-\eps}$.
		Since $\eps>0$ is arbitrary,
		$\mu(A\cap\cdot)$ is supported on a set
		of Hausdorff dimension at most
		$\frac{2}{1+\gamma}$.
\end{enumerate}	
\end{proof}
\pagebreak[2] %% This is here so that the white box marking the end of the proof will not jump to a new page on its own. Remove if necessary.
\begin{rem}
As remarked in the Introduction, the analogous statement for OPRL holds as well. The proof follows the same lines. In fact, since all the
relevant results already exist it is much shorter. Part 1 follows immediately by combining \cite[Theorem 2.2]{last08} with \cite[Theorem
1.1]{LastSimonAC}. Part 2 follows immediately from \cite[Theorem 2.2]{last08} and \cite[Corollary 4.2]{last99}.
\end{rem}

%%%%%%%%%%%%%%%%%%%%%%%%%%%%%%%%%%%%%%%%%%%%%%%%%%%%%%%%%%%%%%%%%%%% Section 3 %%%%%%%%%%%%%%%%%%%%%%%%%%%%%%%%%%%%%%%%%%%%%%%%%%%%%

\section{Sparse Verblunsky Coefficients}

The following is the main technical tool behind Theorem \ref{thm:SingularClock}.

\begin{thm}\label{thm-main}
		Let $\sedecay$ be a sequence of numbers in the open unit disk $\D$ 
		such that $v_\ell \to 0$ as $\ell \to \infty$. 
		Let $\sesparse$ be a strictly increasing sequence of
		integers and let $\mu$ be the measure corresponding to the Verblunsky coefficients
		\begin{equation*}
			\an = \begin{cases}
				v_\el & \text{if } n=N_\el \\
				0 & \text{otherwise}.
			\end{cases}
		\end{equation*}
		Let $K_n$ be the associated CD kernel. If $\sesparse$ is sufficiently sparse (see the remark below) then $K_n$ admits sine kernel asymptotics, namely
		\begin{equation} \label{sinasydef}
			\frac
			{\K(
			\ei{(\theta + \frac{2 \pi a}{n})},
			\ei{(\theta + \frac{2 \pi b}{n})})}
			{\K(\eit, \eit)}
			\conv{n}{\infty}
			e^{i \pi (a-\ob)}
			\frac
			{sin (\pi (a-\ob))}
			{\pi (a-\ob)}
		\end{equation}
		uniformly for $\theta \in [0,2\pi)$ and for $a,b$ in compact subsets of the strip $\abrange$.
	\end{thm}
	\begin{rem}
		By $\sesparse$ being sufficiently sparse we mean that for every $\el\in\N$ there exists an integer $\hatn$ (which depends on
$N_1,...,N_\el$ and on the sequence $\sedecay$), such that $N_{\el+1}$ is at least larger than $\hatn$.
	\end{rem}
	
That \eqref{sinasydef} implies clock behavior follows basically from \eqref{eq-para-cd} and is known as the Freud-Levin-Lubinsky Theorem in
the OPRL setting \cite{freud71,levin08,simon08}. The OPUC analog is presented below in Theorem \ref{thm-clock}. In addition,
\cite[Theorem~12.5.2]{simon09} says that if $\sedecay$ converges to zero and in addition
	$$
	 	\underset{\el \to \infty}{\lim}
	 	\frac{N_{\el+1}}{N_\el} = \infty
	 	\quad \text{and} \quad
	 	\sss{\el=0}{\infty} |v_\el|^2 = \infty,
	$$
	then the measure described in Theorem \ref{thm-main} is purely singular continuous. Thus, Theorem \ref{thm:SingularClock} follows from
this discussion and Theorem \ref{thm-main} above. Accordingly, the rest of this section and Section 4 is devoted to proving Theorem
\ref{thm-main}.

The simplest measure on $\partial \D$ for which sine kernel asymptotics hold is the normalized Lebesgue measure, which corresponds to the Verblunsky coefficients $\an \equiv 0$.
While this follows of course from \cite{lubinsky07}, it is also a direct computation that we
present in Subsection \ref{sec_leb}. The bulk of the proof lies in showing that for any $\sedecay$ it is possible to choose the sequence
$N_\ell$ in such a way that the asymptotics of $K_n$ remain unchanged under a sparse decaying perturbation.
	
Below, $\mu^\Leb := \frac{d\theta}{2\pi}$ denotes the normalized Lebesgue measure on the unit circle, $\mu$ denotes the
perturbed measure appearing in Theorem \ref{thm-main}, and $\ml$ denotes the finitely-perturbed measure corresponding to the Verblunsky
coefficients
	\[
		\an =
		\begin{cases}
			v_j & n=N_j \text{, } j \leq \el \\
			0 	& \text{otherwise}
		\end{cases}.
	\]
	Let $K_n^\Leb,\K,\Kl$ denote the CD kernels of $\mu^\Leb,\mu,\ml$ respectively, and let $\p,\pl{}$ denote the normalized orthogonal
polynomials of $\mu,\ml$ respectively.
	Furthermore, to shorten the formulation of \eqref{sinasydef}, denote
	\begin{align*}
		z_n &:= \ei{(\theta + \frac{2 \pi a}{n})}, \\
		w_n &:= \ei{(\theta + \frac{2 \pi b}{n})}
	\end{align*}
	where $\theta \in [0,2\pi)$ and $a,b \in \abrange$.

We first prove a simple uniform bound on powers of $z_n$ and $w_n$.
	\begin{lem}\label{lem-power-bound}
		For every two integers $m < n$,
		\begin{align*}
			|z_n^m|, |w_n^m| < e^\pi.
		\end{align*}
	\end{lem}
	\begin{proof}
		By the monotonicity of the real exponent, together with $-\imaginary{a} < \frac{1}{2}$, we get
		\begin{align*}
			|z_n^m|
			&=
			\left|e^{im(\theta + \frac{2 \pi a}{n})}\right| \\
			&=
			e^{-\frac{2 \pi m}{n} \imaginary{a}} \\
			&<
			e^\pi.
		\end{align*}
		The same can be shown for $w_n$ using the fact that $-\imaginary{b} < \frac{1}{2}$.
	\end{proof}
	\enskip
	
	\enskip
	Moreover, \eqref{eq:NormEqualityOnCircle} implies another useful inequality 
	for every $z \in \partial\D$ 
	(see \cite[(1.5.27)]{simon09}): 
	\begin{equation}\label{normineq}
	\begin{split}
		z \p_n &= \rho_n \p_{n+1} + \ol{\an} \p_n^* \\
		& \Downarrow \\
		|\p_n| &\leq \frac{|\rho_n|}{1-|\an|} |\p_{n+1}| \\
		& =
		\sqrt {\frac
		{1+|\an|}
		{1-|\an|}}
		|\p_{n+1}|.
	\end{split}
	\end{equation}

\subsection{Asymptotics of the CD kernel for the Lebesgue measure}\label{sec_leb}
	The orthogonal polynomials of the normalized Lebesgue measure $\mu^\Leb$ are $\{z^k\}_{k=0}^\infty$. Therefore
	$$
		K_n^\Leb (z_n, w_n) =
		\sss{k=0}{n-1}z_n^k \ol{w_n^k} =
		\sss{k=0}{n-1}
		e^{2 \pi i k \frac{a-\ob}{n}},
	$$
	and similarly
	$$
		K_n^\Leb (\eit, \eit) =
		\sss{k=0}{n-1}
		e^{ik\theta}\ol{e^{ik\theta}}
		= n.
	$$
	We now calculate the asymptotics of the kernel. Using the formula for the sum of a geometric sequence
	\begin{align*}
		\frac
		{K_n^\Leb(z_n,w_n)}
		{K_n^\Leb(\eit,\eit)}
		& =
		\frac{1}{n}
		\sss{k=0}{n-1}
		e^{2 \pi i k \frac{a-\ob}{n}} \\
		& =
		\frac{1}{n}
		\frac
		{1-e^{2 \pi i (a-\ob)}}
		{1-e^{2 \pi i \frac{a-\ob}{n}}}
		\\
		&=
		e^{i \pi (a-\ob)}
		(e^{-i \pi (a-\ob)} - e^{i \pi (a-\ob)})
		\frac{\frac{1}{n}}
		{1-e^{2 \pi i \frac{a-\ob}{n}}} \\
		& \conv{n}{\infty}
		e^{i \pi (a-\ob)}
		\frac{sin(\pi(a-\ob))}
		{\pi (a-\ob)}.
	\end{align*}

\subsection{Asymptotics of the CD kernel for finitely-perturbed measures}\label{sec_finite_pert}
	In order to prove that the CD kernel of a finitely-perturbed measure $\ml$ has sine kernel asymptotics, we show that it is asymptotically
equivalent to the CD kernel of the Lebesgue measure.
	Since the OPUC do not vanish on the unit circle $\parD$, we may write
	\begin{equation}\label{fin_eq1}
		\frac{\Kl(z_n,w_n)}{\Kl(\eit,\eit)}
		=
		\frac{\Kl(z_n,w_n)}{n \plnleit^2}
		\cdot
		\frac{n \plnleit^2}{\Kl(\eit,\eit)}
	\end{equation}
	while also
	\begin{equation}\label{fin_eq2}
		\frac{n \plnleit^2}{\Kl(\eit,\eit)}
		=
		\frac
		{\sss{k=0}{n-1}\plnleit^2}
		{\sss{k=0}{n-1}\plkeit^2}
		\conv{n}{\infty} 1
	\end{equation}
	because $\pl{k}=\pl{\nl+1}$ eventually (i.e. for every $k \geq \nl+1$).
	So by
\eqref{fin_eq1} and \eqref{fin_eq2}, it suffices to prove that
	\begin{equation}\label{fin_goal}
		\left|
		\frac{\Kl(z_n,w_n)}{n \plnleit^2}
		-
		\frac{K_n^\Leb(z_n,w_n)}{K_n^\Leb(\eit,\eit)}
		\right|
		\conv{n}{\infty} 0.
	\end{equation}
	Indeed,
	\begin{align*}
		\left|
		\frac{\Kl(z_n,w_n)}{n \plnleit^2}
		-
		\frac{K_n^\Leb(z_n,w_n)}{K_n^\Leb(\eit,\eit)}
		\right|
		& =
		\frac{1}{n}
		\left|
		\frac
		{\sss{k=0}{n-1}\pl{k}(z_n)\ol{\pl{k}(w_n)}}
		{\plnleit^2}
		-
		\sss{k=0}{n-1}z_n^k \ol{w_n^k}
		\right| \\
		& \leq
		\frac{1}{n}
		\sss{k=0}{n-1}
		\left|
		\frac
		{\pl{k}(z_n)\ol{\pl{k}(w_n)}}
		{\plnleit^2}
		- z_n^k \ol{w_n^k}
		\right|,
	\end{align*}
	which we partition into two sums
	\begin{equation*}
		\frac{1}{n}
		\sss{k=0}{\nl}
		\left|
		\frac
		{\pl{k}(z_n)\ol{\pl{k}(w_n)}}
		{\plnleit^2}
		- z_n^k \ol{w_n^k}
		\right| \\
		+
		\frac{1}{n}
		\sss{k=\nl+1}{n-1}
		\left|
		\frac
		{\pl{k}(z_n)\ol{\pl{k}(w_n)}}
		{\plnleit^2}
		- z_n^k \ol{w_n^k}
		\right|.
	\end{equation*}
	Since for constant $k$, the term
	$
		\left|
		\frac
		{\pl{k}(z_n)\ol{\pl{k}(w_n)}}
		{\plnleit^2}
		- z_n^k \ol{w_n^k}
		\right|
	$ converges as $n\to\infty$, it is in particular a bounded sequence in $n$. Therefore, as for the finite sum,
	$$
		\frac{1}{n}
		\sss{k=0}{\nl}
		\left|
		\frac
		{\pl{k}(z_n)\ol{\pl{k}(w_n)}}
		{\plnleit^2}
		- z_n^k \ol{w_n^k}
		\right|
		\conv{n}{\infty} 0.
	$$
	To take care of the second sum, note that for every $k>\nl$ and every $z\in\C$
	\[
		\pl{k}(z)=z^{k-\nl-1}\pl{\nl+1}(z),
	\]
	so
	\[
		\left|
		\frac
		{\pl{k}(z_n)\ol{\pl{k}(w_n)}}
		{\plnleit^2}
		- z_n^k \ol{w_n^k}
		\right|
		=
		\left|
		z_n^k
		w_n^k
		\right|
		\cdot
		\left|
		\frac
		{(z_n \ol{w_n})^{-\nl-1}\pl{\nl+1}(z_n)
		\ol{\pl{\nl+1}(w_n)}}
		{\plnleit^2}
		- 1
		\right|.
	\]
	We conclude, using Lemma \ref{lem-power-bound}, that
	\begin{align*}
		\frac{1}{n}
		\sss{k=\nl+1}{n-1}
		\left|
		\frac
		{\pl{k}(z_n)\ol{\pl{k}(w_n)}}
		{\plnleit^2}
		- z_n^k \ol{w_n^k}
		\right|
		& \leq
		\frac{e^{2\pi}}{n}
		\sss{k=\nl+1}{n-1}
		\left|
		\frac
		{(z_n \ol{w_n})^{-\nl-1}\pl{\nl+1}(z_n)
		\ol{\pl{\nl+1}(w_n)}}
		{\plnleit^2}
		- 1
		\right| \\
		& \leq
		e^{2\pi}
		%\\ & \quad \times
		\left|
		\frac
		{(z_n \ol{w_n})^{-\nl-1}\pl{\nl+1}(z_n)
		\ol{\pl{\nl+1}(w_n)}}
		{\plnleit^2}
		- 1
		\right| \\
	\end{align*}
	which converges, by continuity, to
	$$
		e^{2\pi} \left|
		\frac
		{(\eit \ol{\eit})^{(-\nl-1)}\pl{\nl+1}(\eit)
		\ol{\pl{\nl+1}(\eit)}}
		{\plnleit^2}
		-1
		\right|
		= 0
	$$
	as $n\to\infty$, thus proving \eqref{fin_goal}.

%%%%%%%%%%%%%%%%%%%%%%%%%%%%%%%%%%%%%%%%%%%%%%%%%%%%%%%%%%%%%%%%%%%% Section 4 %%%%%%%%%%%%%%%%%%%%%%%%%%%%%%%%%%%%%%%%%%%%%%%%%%%%

\section{Proof of Theorem \ref{thm-main}}

We begin by recursively constructing the Verblunsky coefficients of the fully-perturbed measure. Assume that $\{N_j\}_{j=0}^\el$ are
already chosen. We now pick an integer $\hatn$ large enough so that the following conditions are met:
	\begin{enumerate}\label{conditions}
		\item \label{cond_asy}
		$
			\left|
			\frac
			{\Kl (z_n,w_n)}
			{\Kl (\eit, \eit)}
			-
			e^{i \pi (a-\ob)}
			\frac
			{sin (\pi (a-\ob))}
			{\pi (a-\ob)}
			\right|
			\leq
			\frac{1}{\el}
		$
		for every $n \geq \hatn$, which can be guaranteed by Section \ref{sec_finite_pert}.
		\item \label{cond_cont}
		$
			\frac
			{\left|\pl{\nl+1}(z_n)\right|}
			{\plnleit},
			\frac
			{\left|\pl{\nl+1}(w_n)\right|}
			{\plnleit},
			\frac
			{\left|\pls{\nl+1}(z_n)\right|}
			{\plnleit},
			\frac
			{\left|\pls{\nl+1}(w_n)\right|}
			{\plnleit}
			< 2
		$
		for every $n \geq \hatn$, which can be guaranteed by continuity, and the fact that on the unit circle
		$\left|\pls{\nl+1}(\eit)\right|=\plnleit$.
	\end{enumerate}
	We are now free to pick $\nlp$ as long as it is larger than $\hatn$. Exactly in this sense we mean that the sequence $\sesparse$ in
Theorem \ref{thm-main} should be sufficiently sparse. 
	Our goal is now to prove that the limit
	\begin{equation*}\label{goalgrand}
		\left|
		\frac
		{\K (z_n,w_n)}
		{\K (\eit, \eit)}
		-
		e^{i \pi (a-\ob)}
		\frac
		{sin (\pi (a-\ob))}
		{\pi (a-\ob)}
		\right|
		\conv{n}{\infty} 0
	\end{equation*}
	holds uniformly for $\theta \in [0,2\pi)$ and for $a,b$ in compact subsets of the strip $\abrange$.
	We claim that it suffices to show that
	\begin{equation}\label{goal1}
		\underset{\nlp < n \leq N_{\el+2}}{max}
		\left|
		\frac
		{\K (z_n,w_n)}
		{\K (\eit, \eit)}
		-
		\frac
		{\Kl (z_n, w_n)}
		{\Kl (\eit, \eit)}
		\right|
		\conv{\el}{\infty} 0
	\end{equation}
	uniformly for $\theta \in [0,2\pi)$ and for $a,b$ in compact subsets of the strip $\abrange$. Indeed, let $\eps>0$, and assume
\eqref{goal1} holds. Then there exists $L\in\N$ such that for every $\el > L$,
	\begin{equation*}
		\underset{\nlp < n \leq N_{\el+2}}{max}
		\left|
		\frac
		{\K (z_n,w_n)}
		{\K (\eit, \eit)}
		-
		\frac
		{\Kl (z_n, w_n)}
		{\Kl (\eit, \eit)}
		\right|
		< \frac{\eps}{2}.
	\end{equation*}
	We may assume that $L>\frac{2}{\eps}$, otherwise we just increase $L$ as needed. For every $n>N_{L+1}$, let $\widetilde\el$ be the integer
such that
	$N_{\widetilde\el+1} < n \leq N_{\widetilde\el+2}$. Now
	\begin{align*}
		\left|
		\frac
		{\K (z_n,w_n)}
		{\K (\eit, \eit)}
		-
		e^{i \pi (a-\ob)}
		\frac
		{sin (\pi (a-\ob))}
		{\pi (a-\ob)}
		\right|
		& \leq
		\left|
		\frac
		{\K (z_n,w_n)}
		{\K (\eit, \eit)}
		-
		\frac
		{K_n^{(\widetilde\el)} (z_n,w_n)}
		{K_n^{(\widetilde\el)} (\eit, \eit)}
		\right| \\
		& \quad +
		\left|
		\frac
		{K_n^{(\widetilde\el)} (z_n,w_n)}
		{K_n^{(\widetilde\el)} (\eit, \eit)}
		-
		e^{i \pi (a-\ob)}
		\frac
		{sin (\pi (a-\ob))}
		{\pi (a-\ob)}
		\right| \\
		& \leq
		\frac{\eps}{2} + \frac{1}{\widetilde\el} \\
		& < \eps.
	\end{align*}
	So Theorem \ref{thm-main} follows from \eqref{goal1}. Moreover, because $\nlp < n \leq N_{\el+2}$ implies $\K = \Klp$, \eqref{goal1} is
equivalent to
	\begin{equation}\label{goal2}
		\underset{\nlp < n \leq N_{\el+2}}{max}
		\left|
		\frac
		{\Klp (z_n,w_n)}
		{\Klp (\eit, \eit)}
		-
		\frac
		{\Kl (z_n, w_n)}
		{\Kl (\eit, \eit)}
		\right|
		\conv{\el}{\infty} 0.
	\end{equation}
	We shall now prove that \eqref{goal2} holds uniformly for $\theta \in [0,2\pi)$ and for $a,b$ in compact subsets of the strip $\abrange$.
Notice that
	\begin{align*}\label{summation}
		\left|
		\frac
		{\Klp (z_n,w_n)}
		{\Klp (\eit, \eit)}
		-
		\frac
		{\Kl (z_n, w_n)}
		{\Kl (\eit, \eit)}
		\right|
		& \leq
		\left|
		\frac
		{\Klp (z_n,w_n)}
		{\Klp (\eit, \eit)}
		-
		\frac
		{\Kl (z_n, w_n)}
		{\Klp (\eit, \eit)}
		\right| \\
		& \quad +
		\left|
		\frac
		{\Kl (z_n,w_n)}
		{\Klp (\eit, \eit)}
		-
		\frac
		{\Kl (z_n, w_n)}
		{\Kl (\eit, \eit)}
		\right| \\
		& =
		\left|
		\frac
		{\Klp(z_n, w_n) -
		\Kl(z_n, w_n)}
		{\Klp(\eit, \eit)}
		\right| \\
		& \quad +
		\left|
		\frac{\Kl(z_n,w_n)}
		{\Kl(\eit,\eit)}
		\right|
		\cdot \left|
		\frac
		{\Kl(\eit,\eit)-\Klp(\eit,\eit)}
		{\Klp(\eit, \eit)}
		\right|.
	\end{align*}
	So we can deal with each summand on its own.

\subsection{First summand}\label{sec_fir}
	For every $\nlp < n \leq N_{\el+2}$, we would like to estimate
	\begin{equation*}\label{def_A}
		A_{n,\el} :=
		\left|
		\frac
		{\Klp(z_n, w_n) -
		\Kl(z_n, w_n)}
		{\Klp(\eit, \eit)}
		\right|,
	\end{equation*}
	and show that it converges to zero as $\el\to\infty$.

	Our approach offers a technical simplification to the one found in \cite{breuer11}. There, the analogue of $A_{n,\el}$ was estimated using
the CD formula. Since the CD formula only holds outside of the diagonal ($z \neq w$ when both are real), special care had to be taken
for the denominator of $A_{n,\el}$, as well as the numerator in the case $a=b$. To solve that, a subtle argument for analyticity and Cauchy's
integral formula were used. We found that it is possible to estimate $A_{n,\el}$ directly without invoking the CD formula at all, thus slightly
simplifying the argument. While we only show it here for the OPUC case, our adjustments also work for the OPRL case of \cite{breuer11}.

	Since $\pl{k}=\plp{k}$ for every $k \leq \nlp$, we find that
	\begin{align*}
		A_{n,\el} &= \frac
		{ \left|
		\sss{k=\nlp+1}{n-1}
		\left(
		\plp{k}(z_n) \ol{\plp{k}(w_n)}
		-
		\pl{k}(z_n) \ol{\pl{k}(w_n)}
		\right) \right|}
		{\sss{k=0}{n-1}\plpkeit^2} \\
		& \leq \frac
		{\sss{k=\nlp+1}{n-1}
		\left|
		\plp{k}(z_n) \ol{\plp{k}(w_n)}
		-
		\pl{k}(z_n) \ol{\pl{k}(w_n)}
		\right|}
		{\sss{k=0}{n-1}\plpkeit^2}.
	\end{align*}
	Let us focus on a single term in the numerator, denote
	$$
		A_{n,\el,k}
		:=
		\left|
		\plp{k}(z_n) \ol{\plp{k}(w_n)}
		-
		\pl{k}(z_n) \ol{\pl{k}(w_n)}
		\right|
	$$
	for $\nlp<k<n \leq N_{\el+2}$. From the recursion relation \eqref{recrel} we derive the following at any point $z\in\C$:
	\begin{equation}
	\begin{split}\label{eq-estimate}
		\plp{k}(z) &=
		z^{k-\nlp-1} \plp{\nlp+1}(z) \\
		&=
		z^{k-\nlp-1} \rho^{-1}_{\nlp}
		\left(
		z \plp{\nlp}(z)
		- \ol{\vlp} \plps{\nlp}(z)
		\right) \\
		&=
		z^{k-\nlp-1} \rho^{-1}_{\nlp}
		\left(
		z \pl{\nlp}(z)
		- \ol{\vlp} \pls{\nlp}(z)
		\right) \\
		&=
		\rho^{-1}_{\nlp}
		\left(
		\pl{k}(z)
		- z^{k-\nlp-1} \ol{\vlp} \pls{\nlp}(z)
		\right).
	\end{split}
	\end{equation}
	Therefore,
	\begin{align*}
		\plp{k}(z_n) \ol{\plp{k}(w_n)}
		& =
		\rho_\nlp^{-2} \left[
		\pl{k}(z_n) \ol{\pl{k}(w_n)}
		\right. \\
		& \quad -
		\vlp \pl{k}(z_n)
		\ol{w_n^{k-\nlp-1}
		\pls{\nlp}(w_n)} \\
		& \quad -
		\ol{\vlp \pl{k}(w_n)}
		z_n^{k-\nlp-1}
		\pls{\nlp}(z_n) \\
		& \quad + \left.
		|\vlp|^2 (z_n \ol{w_n})^{k-\nlp-1}
		\pls{\nlp}(z_n)
		\ol{\pls{\nlp}(w_n)}
		\right].
	\end{align*}
	Plugging this into $A_{n,\el,k}$, we get
	\begin{align*}
		A_{n,\el,k}
		& \leq
		(\rho_\nlp^{-2}-1)
		\left|
		\pl{k}(z_n)
		\pl{k}(w_n)
		\right| \\
		& \quad +
		\rho_\nlp^{-2} |\vlp|
		\left(
		\left| \pl{k}(z_n)
		w_n^{k-\nlp-1}
		\pls{\nlp}(w_n)
		\right|
		+
		\left| \pl{k}(w_n)
		z_n^{k-\nlp-1}
		\pls{\nlp}(z_n)
		\right|
		\right) \\
		& \quad +
		\rho_\nlp^{-2} |\vlp|^2
		|z_n w_n|^{k-\nlp-1}
		\left|
		\pls{\nlp}(z_n)
		\pls{\nlp}(w_n)
		\right|.
	\end{align*}
	Since the polynomial sequence $\pl{j}$ is affected by the perturbation only up to $j=\nl+1$, afterwards we have the "free" recursion
formulas
	\begin{align*}
		j &\geq \nl + 1 \\
		&\Downarrow \\
		\pl{j}(z) &= z^{j-\nl-1} \pl{\nl+1}(z), \\
		\pls{j}(z) &= \pls{\nl+1}(z).
	\end{align*}
	We can use these formulas to regress all the $\pl{}$'s back to the last perturbed index $\nl+1$. 

	Applying Lemma \ref{lem-power-bound}, we
see that
	\begin{align*}
		A_{n,\el,k}
		& \leq
		(\rho_\nlp^{-2}-1)
		e^{2\pi}
		\left|
		\pl{\nl+1}(z_n)
		\pl{\nl+1}(w_n)
		\right| \\
		& \quad +
		\rho_\nlp^{-2} |\vlp|
		e^{2\pi}
		\left(
		\left| \pl{\nl+1}(z_n)
		\pls{\nl+1}(w_n)
		\right|
		+
		\left| \pl{\nl+1}(w_n)
		\pls{\nl+1}(z_n)
		\right|
		\right) \\
		& \quad +
		\rho_\nlp^{-2} |\vlp|^2
		e^{2\pi}
		\left|
		\pls{\nl+1}(z_n)
		\pls{\nl+1}(w_n)
		\right|.
	\end{align*}
	Now we may use Condition \ref{cond_cont} from the beginning of Section 4, so we write
	$$
		A_{n,\el,k}
		\leq
		\qlp
		\plnleit^2
	$$
	where
	$$
	\qlp:=
	e^{2\pi} \left(
	4 (\rho_\nlp^{-2}-1)
	+
	8 \rho_\nlp^{-2} |\vlp|
	+
	4 \rho_\nlp^{-2} |\vlp|^2
	\right).
	$$
	Note that $\qlp$ is independent of $n,k,a,b,\theta$, and converges to zero as $\el\to\infty$.
	By \eqref{normineq}, we conclude
	$$
		A_{n,\el,k} \leq
		\qlp \frac
		{1+|\vlp|^2}
		{1-|\vlp|^2}
		\left| \plp{\nlp+1}(\eit) \right|^2.
	$$
	Finally, we plug $A_{n,\el,k}$ back into $A_{n,\el}$
	\begin{align*}
		A_{n,\el}
		& \leq
		\qlp \frac
		{1+|\vlp|^2}
		{1-|\vlp|^2}
		\frac
		{\sss{k=\nlp+1}{n-1}
		\left| \plp{\nlp+1}(\eit) \right|^2}
		{\sss{k=0}{n-1}\plpkeit^2} \\
		& \leq \qlp \frac
		{1+|\vlp|^2}
		{1-|\vlp|^2}
		\conv{\el}{\infty} 0.
	\end{align*}

\subsection{Second summand}\label{sec_sec}
	The second summand is comprised of two parts. Clearly, $\left|
		\frac{\Kl(z_n,w_n)}
		{\Kl(\eit,\eit)}
		\right|$
	is bounded by the asymptotics of the finite perturbation (Condition \ref{cond_asy} from the beginning of Section 4). Namely,
	$
		\left|
		\frac{\Kl(z_n,w_n)}
		{\Kl(\eit,\eit)}
		\right|
		\leq
		\left|
		e^{i \pi (a-\ob)}
		\frac
		{sin (\pi (a-\ob))}
		{\pi (a-\ob)}
		\right|
		+ \frac{1}{\el}
	$.
	We are finally left with the last part
	\begin{equation*}
		\left| \frac
		{\Kl(\eit,\eit)-\Klp(\eit,\eit)}
		{\Klp(\eit, \eit)}
		\right|
	\end{equation*}
	for $\nlp < n \leq N_{\el+2}$, and we want to show that it converges to zero as $\el\to\infty$. That will conclude our proof. But this is a special
case of the first summand (Section \ref{sec_fir}), in which $a=b=0$. We are done.

%%%%%%%%%%%%%%%%%%%%%%%%%%%%%%%%%%%%%%%%%%%%%%%%%%%%%%%%%%%%%%%%%%%% Appendix %%%%%%%%%%%%%%%%%%%%%%%%%%%%%%%%%%%%%%%%%%%%%%%%%%%%

\section{Appendix}

We prove here the following
	\begin{thm}\label{thm-clock}
		Let $\mu$ be a measure on $\parD$ 
		exhibiting sine kernel asymptotics in the sense of \eqref{sinasydef}, 
		uniformly for $\theta \in [0,2\pi)$ 
		and for $a,b$ in compact subsets of the strip $\abrange$. 
		Let $\se{H_n}$ be any corresponding sequence of paraorthogonal polynomials.
		Then for any $e^{i \Theta} \in \partial \D$ and any $j\in\Z$,
		\begin{equation*}
			n \left(\thn{j+1}(\Theta)-\thn{j}(\Theta) \right)
			\conv{n}{\infty}
			2 \pi.
		\end{equation*}
	\end{thm}

	\begin{proof}
		Since $\Theta$ is fixed in the proof, 
		we omit it from the notation $\thn{j}(\Theta)$ throughout.
		Define a sequence of functions
		\begin{equation*}
			f_n(x) = \frac
			{\K(
			\ei{(\thn{j}+\frac{2 \pi x}{n})}
			,\ei{\thn{j}})}
			{\K(\ei{\thn{j}}, \ei{\thn{j}})}.
		\end{equation*}
		By the sine kernel asymptotics, $f_n$ converges to
		\begin{equation*}
			f(x) = \ei{\pi x}
			\frac
			{sin(\pi x)}
			{\pi x}
		\end{equation*}
		uniformly on compact subsets of the strip
		$\{|\imaginary{x}<\frac{1}{2}|\}$.
		
		Let
		$
			a_n = \frac{n}{2\pi}
			\zerodiff
		$
		be the sequence which we want to show converges to 1.
		Suppose, for the sake of contradiction, that
		\begin{equation*}
			\underset{n\to\infty}{\liminf}
			\enskip a_n
			< 1,
		\end{equation*}
		so there exists a subsequence $\{a_{n_k}\}_{k=0}^\infty$
		that converges to $0 \leq L < 1$.
		By the uniform convergence of $f_n$ and the continuity of $f$,
		we conclude that
		\begin{equation*}
			f_{n_k}
			\left( a_{n_k} \right)
			\conv{k}{\infty}
			f(L).
		\end{equation*}
		But
		\begin{equation*}
			f_{n_k}
			\left( a_{n_k} \right)
			=
			\frac
			{K_{n_k}(
			\ei{\thnk{j+1}}, \ei{\thnk{j}}
			)}
			{K_{n_k}(
			\ei{\thnk{j}}, \ei{\thnk{j}}
			)}
			\equiv 0
		\end{equation*}
		while $f(L) \neq 0$ because $L$ is not a nonzero integer,
		which is a contradiction. Therefore,
		\begin{equation*}
			\underset{n\to\infty}{\liminf}
			\enskip a_n
			\geq 1.
		\end{equation*}
		On the other hand, note that $f(1)=0$.
		Due to Hurowitz's theorem,
		there is a sequence
		$\{x_n\}_{n=0}^\infty$
		such that $x_n$ is a zero of $f_n$,
		and $x_n \conv{n}{\infty} 1$.
		But all the zeros of $f_n$ are of the form
		$
		\frac{n}{2 \pi}
		(\thn{m} - \thn{j})
		$
		for some $m \neq j$,
		and thus $a_n$ is the smallest positive zero of $f_n$.
		It follows that for all large enough $n$,
		\begin{align*}
			a_n & \leq x_n \conv{n}{\infty} 1 \\
			& \Downarrow \\
			\underset{n\to\infty}{\limsup}
			\enskip a_n & \leq 1.
		\end{align*}
		We have found that $\underset{n\to\infty}{\lim} a_n = 1$,
		which means that
		\[
			n \zerodiff
			\conv{n}{\infty} 2\pi.
		\]
	\end{proof}

%%%%%%%%%%%%%%%%%%%%%%%%%%%%%%%%%%%%%%%%%%%%%%%%%%%%%%%%%%%%%%%%%%% References %%%%%%%%%%%%%%%%%%%%%%%%%%%%%%%%%%%%%%%%%%%%%%%%%%%%%


\begin{thebibliography}{9}

	\bibitem{last10}
		A.~Avila, Y.~Last and B.~Simon,
		\textit{Bulk universality and clock spacing of zeros for ergodic Jacobi matrices with a.c. spectrum},
		Analysis and PDE \textbf{3}
		(2010), 81--108.

	\bibitem{breuer11}
		J.~Breuer,
		\textit{Sine kernel asymptotics for a class of singular measures},
		J.\ Approx.\ Theory \textbf{163}
		(2011), 1478--1491.

	\bibitem{weissman}
		J.~Breuer and D.~Weissman,
		\textit{Level repulsion for Schr\"odinger operators with singular continuous spectrum},
		J.\ Spect.\ Theory \textbf{9} (2019), 429--451.

	\bibitem{cantero06}
		M.~J.~Cantero, L.~Moral and L.~Vel\'azquez,
		\textit{Measures on the unit circle and unitary truncations of unitary operators},
		J.\ Approx.\ Theory \textbf{139}
		(2006), 430--468.

\bibitem{Castillo} K.~Castillo, F.~Marcell\'an and M~N.~Rebocho, \textit{Zeros of paraorthogonal polynomials and linear spectral
    transformations on the unit circle}, Numer.\ Algorithms \textbf{71} (2016), 699--714.


	

	\bibitem{freud71}
		G.~Freud,
		\textit{Orthogonal Polynomials},
		Pergamon Press, Oxford-New York,
		1971.

\bibitem{Golinskii} L.~Golinskii, \textit{Quadrature formula and zeros of para-orthogonal polynomials on the unit circle} Acta Math.\ Hungar.\
    \textbf{96} (2002), 169--186.

	\bibitem{last99}
		S.~Jitomirskaya and Y.~Last,
		\textit{Power-law subordinacy and singular spectra, I. Half-line operators},
		Acta Math.\ \textbf{183}
		(1999), 171--189.

\bibitem{Jones} W.~B.~Jones, O.~Nj\aa stad and W.~J.~Thron, \textit{Moment theory, orthogonal polynomials, quadrature, and continued fractions
    associated with the unit circle}, Bull.\ London Math.\ Soc.\ \textbf{21} (1989), 113--152.

\bibitem{KN} R.~Killip and I.~Nenciu, \textit{Matrix models for circular ensembles}, Int.\ Math.\ Res.\ Not.\ 2004, no.\ 50, 2665--2701.

\bibitem{KS} R.~Killip and M.~Stoiciu, \textit{Eigenvalue statistics for CMV matrices: from Poisson to clock via random matrix ensembles},
    Duke Math.\ J.\ \textbf{146} (2009), 361--399.


 \bibitem{LastSimonAC} Y.~Last and B.~Simon, \textit{Eigenfunctions, transfer matrices, and absolutely continuous spectrum of one-dimensional
     Schrödinger operators}, Invent.\ Math.\ \textbf{135} (1999), 329--367.

	\bibitem{last08}
		Y.~Last and B.~Simon,
		\textit{Fine structure of the zeros of orthogonal polynomials, IV. A priori bounds and clock behavior},
		Comm.\ Pure Appl.\ Math.\ \textbf{61}
		(2008), 486--538.

	\bibitem{lubinsky07}
		E.~Levin and D.~S.~Lubinsky,
		\textit{Universality Limits Involving Orthogonal Polynomials on the Unit Circle},
		Computational Methods and Function Theory \textbf{7}
		(2007), 543--561.

	\bibitem{levin08}
		E.~Levin and D.~S.~Lubinsky,
		\textit{Applications of universality limits to zeros and reproducing kernels of orthogonal polynomials},
		J.\ Approx.\ Theory \textbf{150.1}
		(2008), 69--95.



\bibitem{LubinskyRev} D.~Lubinsky, \textit{An update on local universality limits for correlation functions generated by unitary ensembles},
    SIGMA \textbf{12} (2016), 078, 36pp.

\bibitem{Martinez} A.~Martinez-Finkelshtein, A.~Sri Ranga and D.~Veronese, \textit{Extreme zeros in a sequence of paraorthogonal polynomials
    and bounds for the support of the measure}, Math.\ Comp.\ \textbf{87} (2018), 261--288.

\bibitem{MSS} A.~Martinez-Finkelshtein, B.~Simanek and B.~Simon \textit{Poncelet's Theorem, paraorthogonal polynomials and the numerical
    range of compressed multiplication operators}, Adv.\ in Math.\ \textbf{349} (2019) 992--1035.


\bibitem{rogers} C.~A.~Rogers, \textit{Hausdorff Measures}, London, Cambridge University Press 1970.

\bibitem{Simanek1} B.~Simanek, \textit{Zeros of non-Baxter paraorthogonal polynomials on the unit circle}, Constr.\ Approx.\ \textbf{35}
    (2012), 107--121.

\bibitem{Simanek2} B.~Simanek, \textit{An electrostatic interpretation of the zeros of paraorthogonal polynomials on the unit circle}, SIAM
    J.\ Math.\ Anal.\ \textbf{48} (2016), 2250--2268.


\bibitem{Simanek3} B.~Simanek, \textit{Zero spacing of paraorthogonal polynomials on the unit circle}, Preprint, arXiv 1907.01604v1.


\bibitem{simon06} B.~Simon, \textit{Fine structure of the zeros of orthogonal polynomials on the unit circle, I. A tale of two pictures},
    Electronic Transactions on Numerical Analysis \textbf{25} (2006),268--328.

\bibitem{simon07Rankone} B.~Simon, \textit{Rank one perturbations and the zeros of paraorthogonal polynomials on the unit circle}, J.\ Math.\
    Anal.\ Appl.\ \textbf{329} (2007),
 376--382.

	\bibitem{simon08}
		B.~Simon,
		\textit{The Christoffel-Darboux Kernel},
		Proc.\ Sympos.\ Pure Math.\ \textbf{79},
		Amer.\ Math.\ Soc.,
		Providence, RI, 2008.

	\bibitem{simon09}
		B.~Simon,
		\textit{Orthogonal Polynomials on the Unit Circle},
		American Mathematical Society Colloquium Publications \textbf{54},
		American Mathematical Society,
		Providence, RI,
		2009.


	\bibitem{stoiciu} M.~Stoiciu, \textit{The statistical distribution of the zeros of random paraorthogonal polynomials on the unit circle}, J.\
    Approx.\ Theory \textbf{139} (2006), 29--64.

	\bibitem{wong07}
		M.~Wong,
		\textit{First and second kind paraorthogonal polynomials and their zeros},
		J. Approx. Theory \textbf{146.2}
		(2007), 282--293.

	\bibitem{zlatos04}
		A.~Zlato\v{s},
		\textit{Sparse potentials with fractional Hausdorff dimensions},
		J. Funct. Anal \textbf{207}
		(2004), 216--252.


\end{thebibliography}
\end{document}